\documentclass[a4paper, 11pt]{amsart}
\usepackage[letterpaper, margin=.99in]{geometry}
\usepackage{latexsym}
\usepackage{amssymb}
\usepackage{amsmath}
\usepackage{amsfonts}
\usepackage{amsthm}
\usepackage{color}
\usepackage{mathtools}
\usepackage{relsize}
            \DeclareFontFamily{U}{wncy}{}
            \DeclareFontShape{U}{wncy}{m}{n}{%
               <5>wncyr5%
               <6>wncyr6%
               <7>wncyr7%
               <8>wncyr8%
               <9>wncyr9%
               <10>wncyr10%
               <11>wncyr10%
               <12>wncyr6%
               <14>wncyr7%
               <17>wncyr8%
               <20>wncyr10%
               <25>wncyr10}{}

\newtheorem{thm}{Theorem}[section]

\newtheorem{lem}[thm]{Lemma}

\newtheorem{cor}[thm]{Corollary}
\newtheorem{prop}[thm]{Proposition}
\theoremstyle{definition}

\newtheorem{remark}{Remark}

\newtheorem{example}{Example}

\newcommand{\R}{\mathbb R}
\newcommand{\N}{\mathbb N}
\newcommand{\Z}{\mathbb Z}
\newcommand{\Q}{{\mathbb Q}}
\newcommand{\C}{\mathbb C}

\def\al{\alpha}

\def\la{\lambda}

\begin{document}

\raggedbottom

\title[Action of $H(\lambda_4)$ on Imaginary Fields]{The Hecke group $H(\lambda_4)$ acting on imaginary quadratic number fields}

\author{Abdulaziz Deajim}
\address[]{Department of Mathematics, King Khalid University,
P.O. Box 9004, Abha, Saudi Arabia} \email{deajim@kku.edu.sa, deajim@gmail.com}

\date{\today}

\begin {abstract}
Let $H(\lambda_4)$ be the Hecke group $\langle x,y\,:\, x^2=y^4=1 \rangle$ and, for a square-free positive integer $n$, consider the subset $\mathbb{Q}^*(\sqrt{-n})=\left\{(a+\sqrt{-n})/c \, | \, a,b=(a^2+n)/c \in \mathbb{Z},\, c\in 2\mathbb{Z} \right\}$ of the quadratic imaginary number field $\mathbb{Q}(\sqrt{-n})$. Following a line of research in the relevant literature, we study properties of the action of $H(\lambda_4)$ on $\mathbb{Q}^*(\sqrt{-n})$. In particular, we calculate the number of orbits arising from this action for every such $n$. Some illustrative examples are also given.
\end {abstract}
\maketitle

\section{{\bf Introduction}}\label{intro}
\subsection{Preliminaries}
For a positive real number $\la$, E. Hecke intoduced in \cite{Hecke} the group $H(\la)$ generated by the linear transformations $x: z\mapsto -1/z$ and $w: z\mapsto z+\la$. He further showed that if $\la=\la_k=2\cos(\pi/k)$ for an integer $k\geq 3$, then $H(\la_k)$ is Fuchsian (i.e. a discrete subgroup of $\mbox{PSL}(2,\R)$) (see \cite{Hecke} and \cite{SBC}). The groups $H(\la_k)$ are justifiably called {\it Hecke groups}. Letting $y=xw:z\mapsto -1/(z+\la_k)$, it can be shown that $H(\la_k)$ is generated by $x$ and $y$, so $H(\la_k)=\langle x,y\,:\, x^2=y^k=1 \rangle$, and $H(\la_k)$ is isomorphic to a free product $C_2 \ast C_k$ of the cyclic groups of orders 2 and $k$ (see \cite{CD}). Such a presentation of $H(\la_k)$ is by no means unique. The actions of Hecke groups on many discrete structures play important roles in various branches of mathematics (see \cite{CD} for more details).

For $k=3$, it is known (see \cite{HM}) that $H(\la_3)$ is the modular group $\mbox{PSL}(2, \Z)$. The action of the modular group on real quadratic number fields was extensively studied, see for instance \cite{HMM}, \cite{MZ}, \cite{M1}, \cite{M2}, and \cite{MAY}. On the other hand, some aspects of the action of the modular group (resp. a particular subgroup of the modular group) on imaginary quadratic number fields were also studied, see \cite{AD} (resp. \cite{AM}).

For $k=4$, we set $H=H(\la_4)$ in this paper. Following \cite{MA}, we choose the following presentation for $H$:
$$\mbox{$H=\langle x,y\,:\, x^2 = y^4 =1 \rangle$ with $x:z \mapsto \cfrac{-1}{2z}$ and $y: z \mapsto \cfrac{-1}{2(z+1)}$}\cdot$$
The action of $H$ on real quadratic number fields was also studied (see for instance \cite{MHM} and \cite{MA}). In this paper, we consider the action of $H$ on {\it imaginary} quadratic number fields. In particular, for a square-free positive integer $n$, we study the action of $H$ on the following subset of $\Q(\sqrt{-n})$:
$$\Q^*(\sqrt{-n}):=\left\{\,\frac{a+\sqrt{-n}}{c}\in \Q(\sqrt{-n})\; | \; a,b=\frac{a^2+n}{c} \in \Z,\, c\in 2\Z \,\right\}.$$
We observe that the set $\Q^*(\sqrt{-n})$ is closed under complex conjugation since, for \linebreak $\al=\cfrac{a+\sqrt{-n}}{c}\in \Q^*(\sqrt{-n})$, we have $\overline{\al}=\cfrac{-a+\sqrt{-n}}{-c}$. Furthermore, such an $\al$ and its conjugate are roots of the quadratic polynomial $ct^2-2at +b$ with discriminant $4a^2-4bc=-4n$. In fact, there is a two-to-one correspondence from the set $\Q^*(\sqrt{-n})$ to the set of such polynomials.

\subsection{Summary of Contributions}
After displaying and proving some interesting preliminary results in Sections 2 and 3 on the action of $H$ on $\Q^*(\sqrt{-n})$, we utilize them to prove the main result of the paper (Theorem \ref{main}), which aims at precisely calculating the number of orbits resulting from this action. We present in Section 4 some examples to illustrate various computational aspects that arise when applying Theorem \ref{main} to concrete cases.
%%%%%%%%%%%%%%%%%%%%%%%%%%%%%%%%%%%%%%%%%%%%%%%%%%%%%%

\section{Lemmas}

Throughout the paper, unless a particular value of $n$ is specified, $n$ denotes a square-free positive integer.

\begin{lem}\label{a,n,b,c}
For $\cfrac{a+\sqrt{-n}}{c}\in \Q^*(\sqrt{-n})$, we have:
\begin{itemize}
\item[(i)] $a$ is odd if and only if $n$ is odd.
\item[(ii)] $b$ and $c$ have the same sign.
\end{itemize}
\end{lem}

\begin{proof}
Clear.
\end{proof}

%\begin{lem}\label{fixed}
%\textcolor[rgb]{0.00,0.07,1.00}{The only elements of $\C$ fixed by $x$ are $\sqrt{-2}/2, \, \sqrt{-2}/(-2)\in \Q^*(\sqrt{-2})$, and the only elements fixed by $y$ are $(-1+i)/2, \, (1+i)/(-2)\in \Q^*(\sqrt{-1})$.}
%\end{lem}

%\begin{proof}
%A complex number $z$ fixed by $x$ satisfies $-1/2z=z$ and, thus, $z$ is a solution of the quadratic equation $2z^2+1=0$; that is $z\in \{\pm \sqrt{-2}/2\}$.  While a complex number $z$ fixed by $y$ satisfies $-1/2(z+1)=z$ and, thus, $z$ is a solution of the quadratic equation $2z^2+2z+1=0$; that is $z\in \{(-1\pm i)/2\}$.
%\end{proof}

Following \cite{AD}, for $\al = \cfrac{a+\sqrt{-n}}{c} \in \Q^*(\sqrt{-n})$, we sometimes use the notations $a_\al$ for $a$, $b_\al$ for $b$, and $c_\al$ for $c$ for obvious reasons (see the lemma below).

\begin{lem}\label{table}
For $\al=\cfrac{a+\sqrt{-n}}{c}\in \Q^*(\sqrt{-n})$, the effect of the action of $t\in \{x, y, y^2, y^3\}$ on $\al$ is summarized in the table below:\\
\begin{center}
\begin{tabular}{c|| c c c c c}
\hline
 $t$ & $a_{t(\al)}$ & \quad & $b_{t(\al)}$ & \quad & $c_{t(\al)}$\\
\hline
$x$ & $-a$ & $\quad$ & $c/2$ & $\quad$ & $2b$\\
$y$ & $-a-c$ & $\quad$ & $c/2$ & $\quad$ & $2(2a+b+c)$\\
$y^2$ & $-3a-2b-c$ & $\quad$ & $2a+b+c$ & $\quad$ & $4a+4b+c$\\
$y^3$ & $-a-2b$ & $\quad$ & $(4a+4b+c)/2$ & $\quad$ & $2b$\\
\hline
\end{tabular}
\end{center}
\end{lem}
\vspace{.2cm}
\begin{proof}
With the use of the equality $bc=a^2+n$ we have:
$$x(\al)=\cfrac{-1}{2\al}=\cfrac{-1}{2}\left(\cfrac{c}{a+\sqrt{-n}}\right)=\cfrac{-1}{2}\left(\cfrac{c(a-\sqrt{-n})}{a^2+n}\right)
=\cfrac{-a+\sqrt{-n}}{2b}.$$
So $a_{x(\al)}=-a$, $c_{x(\al)}=2b$, and $b_{x(\al)}=(a_{x(\al)}^2+n)/c_{x(\al)}=((-a)^2+n)/2b=c/2$ as claimed. We also have:
\begin{align*}
y(\al)&=\cfrac{-1}{2(\al+1)}=\cfrac{-c}{2(a+c+\sqrt{-n})}=\cfrac{-c(a+c-\sqrt{-n})}{2(a^2+2ac+c^2+n)}\\
&=\cfrac{-c(a+c-\sqrt{-n})}{2(bc+2ac+c^2)}=\cfrac{-a-c+\sqrt{-n}}{2(2a+b+c)}.
\end{align*}
So, $a_{y(\al)}=-a-c$, $c_{y(\al)}=2(2a+b+c)$, and
\begin{align*}
b_{y(\al)}&= \cfrac{a_{y(\al)}^2+n}{c_{y(\al)}}=\cfrac{(-a-c)^2+n}{2(2a+b+c)}\\
&= \cfrac{a^2+2ac+c^2+n}{2(2a+b+c)}=\cfrac{bc+2ac+c^2}{2(2a+b+c)}\\
&=\cfrac{c(b+2a+c)}{2(2a+b+c)}=c/2
\end{align*}
as claimed. Now the rest of the proof is direct recursive computations using the above formulae. For instance, $a_{y^2(\al)}=-a_{y(\al)}-c_{y(\al)}=-(-a-c)-2(2a+b+c)=-3a-2b-c$, and so on.
\end{proof}

%\begin{proof}
%With the use of $bc=a^2+n$, direct computations show the following (from which we can extract the values $a_{t(\al)}$ and $c_{t(\al)}$ for $t=x,y,y^2$, and $y^3$): \textcolor[rgb]{1.00,0.00,0.00}{fix}
%\begin{align*}
%x(\al)&=\cfrac{-1}{2(\al)}=\cfrac{-1}{2}\left(\cfrac{c}{a+\sqrt{-n}}\right)=\cfrac{-1}{2}\left(\cfrac{c(a-\sqrt{-n})}{a^2+n}\right)
%=\cfrac{a-\sqrt{-n}}{-2b}=\cfrac{-a+\sqrt{-n}}{2b},\\
%y(\al)&=\cfrac{-1}{2(\al+1)}=\cfrac{-c}{2(a+c+\sqrt{-n})}=\cfrac{-c(a+c-\sqrt{-n})}{2(a^2+2ac+c^2+n)}\\
%&=\cfrac{-c(a+c-\sqrt{-n})}{2(bc+2ac+c^2)}=\cfrac{-a-c+\sqrt{-n}}{2(2a+b+c)},\\
%y^2(\al)&=\cfrac{-1}{2(y(\al)+1)}=\cfrac{-(2a+b+c)}{-a-c+2(2a+b+c)+\sqrt{-n}}\\
%&= \cfrac{-(2a+b+c)(3a+2b+c-\sqrt{-n})}{9a^2+4b^2+c^2+12ab+6ac+4bc+n}\\
%&= \cfrac{-6a^2-2b^2-c^2-7ab-5ac-3bc+(2a+b+c)\sqrt{-n}}{8a^2+4b^2+c^2+12ab+6ac+5bc}\\
%&=\cfrac{(-3a-2b-c)(2a+b+c)+(2a+b+c)\sqrt{-n}}{(4a+4b+c)(2a+b+c)}=\cfrac{(-3a-2b-c)+\sqrt{-n}}{4a+4b+c},\\
%y^3(\al)&=\cfrac{-1}{2(y^2(\al)+1)}=\cfrac{-(4a+4b+c)}{2(a+2b+\sqrt{-n})}=\cfrac{-(4a+4b+c)(a+2b-\sqrt{-n})}{2(a^2+4b^2+4ab+n)}\\
%&=\cfrac{-(4a+4b+c)(a+2b-\sqrt{-n})}{2b(4a+4b+c)}=\cfrac{-a-2b+\sqrt{-n}}{2b}.
%\end{align*}
%Finally, calculating $\cfrac{a_{t(\al)}^2+n}{c_{t(\al)}}\,$ for $t=x, y, y^2$, and $y^3$, the values $b_{t(\al)}$ (as in the table) are verified in a similar manner.
%\end{proof}

\begin{remark}
It follows from Lemma \ref{table} that $\Q^*(\sqrt{-n})$ is an $H$-set. To see this, we only need to show that $\Q^*(\sqrt{-n})$ is stable under the action of $H$ (because $H$ acts naturally on $\C$). For this, it suffices to show that $\Q^*(\sqrt{-n})$ is stable under the action of the two generators of $H$. By means of Lemma \ref{table}, direct computations show that, for $\al \in \Q^*(\sqrt{-n})$ and $t=x,y$, we have $a_{t(\al)}^2+n=b_{t(\al)} c_{t(\al)}$ and $c_{t(\al)}$ is even, which imply that $t(\al)\in \Q^*(\sqrt{-n})$.
\end{remark}

\begin{lem}\label{same sign}
Under the action of $H$ on $\Q^*(\sqrt{-n})$, the denominators of all elements of the same orbit are of the same sign.
\end{lem}

\begin{proof}
This follows directly from Lemma \ref{table}, taking into account that the signs of $b_{t(\al)}$ and $c_{t(\al)}$ are the same for $t=x,y$ (Lemma \ref{a,n,b,c}(ii)).
\end{proof}

\vspace{.2cm}
Recall the following classical terminology. An element $\al\in \Q^*(\sqrt{-n})$ is said to be {\it purely imaginary} if $a_\al=0$. We can thus have the following observation, where we use the usual notation $d(m)$ for the number of positive divisors of $m\in \N$.

\begin{lem}\label{norm zero}
$\Q^*(\sqrt{-n})$ contains purely imaginary elements if and only if $n$ is even, in which case the number of purely imaginary elements is equal to $2d(n/2)$.
\end{lem}

\begin{proof}
Suppose that $\Q^*(\sqrt{-n})$ contains a purely imaginary element $\al$. Then $\al$ must be of the form $\sqrt{-n}/c$ with $b=n/c\in \Z$. Since $c$ is even, so must $n$ be. Conversely, if $n$ is even, then $\sqrt{-n}/2\in \Q^*(\sqrt{-n})$ is purely imaginary. Now, let $n$ be even with $n=2n'$, $n'\in\N$. The set of purely imaginary elements in $\Q^*(\sqrt{-n})$ are $$\{\sqrt{-n}/c\;|\: \mbox{$c$ is even and $c$ divides $n$}\}=\{\sqrt{-n}/2c'\;|\; \mbox{$c'\in \Z$ and $c'$ divides $n'$}\}.$$
The cardinality of the above set is obviously $2d(n')$, taking into consideration positive and negative divisors of $n'$.
\end{proof}

\vspace{.2cm}
An element $\al \in \Q^*(\sqrt{-n})$ is said to be {\it totally positive} (resp. {\it totally negative}) if $a_\al c_\al >0$ (resp. $a_\al c_\al <0$). It is obvious that an element of $\Q^*(\sqrt{-n})$ is either purely imaginary, totally positive, or totally negative. We observe that the set of totally positive elements of $\Q^*(\sqrt{-n})$ lies in the open right-half of the complex plane, the set of totally negative elements lies in the open left-half of the complex plane, and the set of purely imaginary elements lies on the imaginary axis.

The following lemma studies the effects of $x, y, y^2$, and $y^3$ on elements of $\Q^*(\sqrt{-n})$ in each of the three states.

\begin{lem}\label{negative quadrable}
Let $\al = \cfrac{a+\sqrt{-n}}{c}\in \Q^*(\sqrt{-n})$. We have the following:
\begin{itemize}
\item[1.] $\al$ is purely imaginary if and only if $x(\al)$ is purely imaginary.
\item[2.] $\al$ is totally positive if and only if $x(\al)$ is totally negative.
\item[3.] If $\al$ is totally positive or purely imaginary, then $y(\al)$, $y^2(\al)$, and $y^3(\al)$ are all totally negative.
\item[4.] $\al, y(\al), y^2(\al)$, and $y^3(\al)$ are all totally negative if and only if $\al$ is totally negative, $|a|<|c|$, $|a|<2|b|$, and $3|a|<2|b|+|c|$.
\end{itemize}
\end{lem}

\begin{proof}
\begin{itemize}
\item[1.] From Lemma \ref{table}, $a_{x(\al)}=-a_\al$.
\item[2.] This follows from Lemma \ref{table}, as $a_{x(\al)}$ and $a_{\al}$ have opposite signs, while $c_{x(\al)}$ and $c_{\al}$ have the same sign.
\item[3.] Recall from Lemma \ref{a,n,b,c} that $b$ and $c$ have the same sign. If $\al$ is totally positive, then either $a,b,c>0$ or $a,b,c<0$. In both cases, it follows directly from Lemma \ref{table} that $y(\al)$, $y^2(\al)$, and $y^3(\al)$ are all totally negative. Similarly, if $\al$ is purely imaginary then $a=0$ and $b,c<0$ or $b,c>0$. In either case, it follows directly from Lemma \ref{table} that $y(\al)$, $y^2(\al)$, and $y^3(\al)$ are all totally negative in this case as well.
\item[4.] Assume that $\al$ is totally negative, $|a|<|c|$, $|a|<2|b|$, and $3|a|<2|b|+|c|$. Suppose first that $a>0$ and $b,c<0$. So, $a<-c$, $a<-2b$, and $3a<-2b-c$. By Lemma \ref{table}, we deduce that $a_{y(\al)}=-a-c>0$, $a_{y^2(\al)}=-3a-2b-c>0$, and $a_{y^3(\al)}=-a-2b>0$. Note also that $c_{y(\al)}<0$ as $b_{y(\al)}=c/2<0$, $c_{y^2(\al)}<0$ as $b_{y^2(\al)}=2a+b+c<0$, and $c_{y^3(\al)}=2b<0$. It thus follows that $y(\al)$, $y^2(\al)$, and $y^3(\al)$ are all totally negative. The case when $a<0$ and $b,c>0$ is handled similarly.

Conversely, assume that $\al, y(\al), y^2(\al)$, and $y^3(\al)$ are all totally negative. Suppose that $a>0$ and $b,c<0$. Since $y(\al)$ is totally negative and $b_{y(\al)}=c/2<0$, it follows that $-a-c =a_{y(\al)}>0$. So, $|a|=a<-c=|c|$. Since $b_{y(\al)}$ and $c_{y(\al)}$ have the same sign and $b_{y(\al)}=c/2<0$, $2(2a+b+c)=c_{y(\al)}<0$. Since $y^2(\al)$ is totally negative and $b_{y^2(\al)}=2a+b+c=c_{y(\al)}/2<0$, $-3a-2b-c=a_{y^2(\al)}>0$. Thus, $3|a|=3a<-2b-c=2|b|+|c|$. Finally, since $y^3(\al)$ is totally negative and $c_{y^3(\al)}=2b<0$, it follows that $-a-2b=a_{y^3(\al)}>0$. Thus, $|a|=a<-2b=2|b|$. The case when $a<0$ and $b,c>0$ is handled similarly.
\end{itemize}
\end{proof}

\begin{example}\hfill
\begin{itemize}
\item[1.] For the totally positive element $\al_1=\cfrac{2+\sqrt{-2}}{2}\in \Q^*(\sqrt{-2})$, we see that $y(\al_1)=\cfrac{-4+\sqrt{-2}}{18}$, $y^2(\al_1)=\cfrac{-14+\sqrt{-2}}{22}$, $y^3(\al_1)=\cfrac{-6+\sqrt{-2}}{6}$, which are all totally negative.
\item[2.] For the purely imaginary element $\al_2=\sqrt{-2}/2\in \Q^*(\sqrt{-2})$, we see that $y(\al_2)=\cfrac{-2+\sqrt{-2}}{6}$, $y^2(\al_2)=\cfrac{-4+\sqrt{-2}}{6}$, and $y^3(\al_2)=\cfrac{-2+\sqrt{-2}}{2}$, which are all totally negative.
\item[3.] For the totally negative element $\al_3=\cfrac{1+\sqrt{-5}}{-2} \in \Q^*(\sqrt{-5})$ we see that $a_{\al_3}=1$, \linebreak $b_{\al_3}=-3$, $c_{\al_3}=-2$ which satisfy the required inequalities of part 4 of Lemma \ref{negative quadrable}. Clearly, $\displaystyle{y(\al_3)=\cfrac{1+\sqrt{-5}}{-6}}$, $y^2(\al_3)=\cfrac{5+\sqrt{-5}}{-10}$, and $y^3(\al_3)=\cfrac{5+\sqrt{-5}}{-6}$ are all totally negative.
\end{itemize}
\end{example}

\begin{remark}\hfill
\begin{itemize}
\item[1.] For any set of four elements $\{\al, y(\al), y^2(\al), y^3(\al)\}\subseteq \Q^*(\sqrt{-n})$, we see from parts 3 and 4 of Lemma \ref{negative quadrable} that either all four elements are totally negative, one is totally positive and the other three are totally negative, or one is purely imaginary and the other three are totally negative.

\item[2.] We show here that the three conditions $|a|<|c|$, $|a|<2|b|$, and $3|a|<2|b|+|c|$ in part 4 of Lemma \ref{negative quadrable} are all necessary and so none of them can be dropped. The totally negative element $\al_1=\cfrac{2+\sqrt{-2}}{-2}\in \Q^*(\sqrt{-2})$ has the two properties $|a|<2|b|$ and $3|a|<2|b|+|c|$, but not $|a|<|c|$ (in fact, $|a|=|c|$). So, $y(\al_1)$ is not totally negative. The totally negative element $\al_2=\cfrac{4+\sqrt{-2}}{-6}\in \Q^*(\sqrt{-2})$ has the two properties $|a|<|c|$ and $|a|<2|b|$, but not $3|a|<2|b|+|c|$ (in fact, $3|a|=2|b|+|c|$). So, $y^2(\al_2)$ is not totally negative. The totally negative element $\al_3=\cfrac{2+\sqrt{-2}}{-6}\in \Q^*(\sqrt{-2})$ has the two properties $|a|<|c|$ and $3|a|<2|b|+|c|$, but not $|a|<2|b|$ (in fact, $|a|=2|b|$). So, $y^3(\al_3)$ is not totally negative.
\end{itemize}
\end{remark}

\begin{lem}\label{positive}
Under the action of $H$ on $\Q^*(\sqrt{-n})$, every orbit contains a totally positive element.
\end{lem}

\begin{proof}
Let $\al\in \Q^*(\sqrt{-n})$ and $\al^H=\{\beta\,|\, \mbox{$h(\al)=\beta$ for some $h\in H$}\} $ be the orbit containing $\al$. If $\al$ is totally positive, then there is nothing to prove. If $\al$ is totaly negative then $x(\al)\in \al^H$ is totally positive (Lemma \ref{negative quadrable}(2)). If $\al$ is purely imaginary, then $y(\al)\in \al^H$ is totally negative and so $xy(\al)\in \al^H$ is totally positive (Lemma \ref{negative quadrable}(2,3)).
\end{proof}

%%%%%%%%%%%%%%%%%%%%%%%%%%%%%%%%%%%%%%%%%%%%%%%%%%%%%%%%%%%%%%%%%%%

\section{Orbits in $\Q^*(\sqrt{-n})$ under the action of $H$}

For any $\al \in \Q^*(\sqrt{-n})$, denote the quadruplet $\{\al, y(\al), y^2(\al), y^3(\al)\}$ by $Q_\al$ for short. Note that $Q_\al$ is stable under the action of the cyclic group $<y>$ generated by $y$, that is $Q_\al=Q_{y^i(\al)}$ for $i=1,2,3$.
If all elements of $Q_\al$ are totally negative, then we call $Q_\al$ a {\it totally negative quadruplet}. Set
$$TN(-n):=\{Q_\al \subseteq \Q^*(\sqrt{-n}) \,|\, \mbox{$Q_\al$ is a totally negative quadruplet}\}.$$ By part 4 of Lemma \ref{negative quadrable} we have
$$TN(-n)=\{Q_\al \subseteq \Q^*(\sqrt{-n})\,|\, \mbox{$a_\al c_\al<0, |a_\al|<|c_\al|, |a_\al|<2|b_\al|, 3|a_\al|<2|b_\al|+|c_\al|$}\}.$$
If $TN(-n)$ is not empty, then it follows from Lemma \ref{same sign} that we can partition $TN(-n)$ into the disjoint union $TN(-n)=TN_{a>0}(-n)\bigcup TN_{a<0}(-n)$, where
$$TN_{a>0}(-n):=\{Q_\al \subseteq \Q^*(\sqrt{-n})\,|\, \mbox{$a_\al>0, c_\al<0, a_\al <|c_\al|, a_\al <2|b_\al|, 3 a_\al<2|b_\al|+|c_\al|$}\}$$
and
$$TN_{a<0}(-n):=\{Q_\al \subseteq \Q^*(\sqrt{-n})\,|\, \mbox{$a_\al<0, c_\al>0, |a_\al|< c_\al, |a_\al|<2 b_\al, 3|a_\al|<2 b_\al+c_\al$}\}.$$
Moreover, it is easy to see that there is a bijection between the two sets $TN_{a>0}(-n)$ and $TN_{a<0}(-n)$ and, therefore, $$|TN(-n)|=|TN_{a>0}(-n)|+|TN_{a<0}(-n)|=2\,|TN_{a>0}(-n)|.$$
For an $\al$ with $Q_\al \in TN_{a>0}(-n)$, we call $(a_\al, -b_\al, -c_\al)\in \N^3$ the {\it signature} of $\al$, and we set $$S_{a>0}(-n):=\{(a_\al, -b_\al, -c_\al)\in \N^3\,|\, Q_\al \in TN_{a>0}(-n)\}.$$
It is clear that we can also write $S_{a>0}(-n)$ as
$$S_{a>0}(-n)=\{(a,-b,-c)\in \N^3\,|\, bc=a^2+n, a<-c, a<-2b, 3a<-2b-c\}.$$

It shall be apparent shortly that in order to count the number of orbits in $\Q^*(\sqrt{-n})$, we will need to calculate the cardinality of $S_{a>0}(-n)$, and for this endeavor it is very helpful to find the maximum value that $a$ can possibly take in order for $(a,-b,-c)$ to belong to $S_{a>0}(-n)$. Based on Lemma \ref{bound} below, we see that such an $a$ can not exceed $n$ if $n$ is odd, and it can not exceed $n/2$ if $n$ is even. Lemma \ref{a,n} is needed in proving Lemma \ref{bound}.

%\begin{lem}\label{H_y}
%$S_{a>0}(-n)$ is an $H_y$-set if $S_{a>0}(-n)\neq \varnothing$, where $H_y$ is the cyclic subgroup of $H$ generated by $y$.
%\end{lem}

\begin{lem}\label{a,n}
Let $(a,-b,-c)\in S_{a>0}(-n)$ with $c=2c'$.
\begin{itemize}
\item[(i)] If $n$ is odd and $|b|\leq n$ or $|c'|\leq n$, then $a\leq n$.
\item[(ii)] If $n$ is even and $|b|\leq n/2$ or $|c'| \leq n/2$, then $a\leq n/2$.
\end{itemize}
\end{lem}

\begin{proof}
We will only prove part (i), as the proof of part (ii) follows a similar argument. Assume that $n$ is odd. Denote $|b|$, $|c|$, and $|c'|$ by $B$, $C$, and $C'$, respectively. Notice that as $b$ and $c$ are of the same sign, $bc=BC$. We then have $a^2+n=BC=2BC'$, $a<2C'$, $a<2B$, and $3a<2B+2C'$. Due to the apparent symmetry between $B$ and $C'$, it suffices to prove that if $B\leq n$ then $a\leq n$ (the proof of the other statement is similar). Assume that $B\leq n$.
\begin{itemize}
\item[Case (1):] Assume that $B=n$. Since $a^2=n(C-1)$, $n$ divides $a^2$. Since $n$ is square-free, $n$ divides $a$ and so $a\geq n$. If $a>n$, then $a/n >1$. Since $n$ and $a$ are both odd (Lemma \ref{a,n,b,c}), $a/n\geq 3$. So, $a\geq 3n$. But $a<2B=2n$, which is a contradiction. Thus, $a=n$ in this case.
\item[Case (2):] Assume that $B\leq (n+1)/2$. Then we have $a<2B\leq n+1$ and so $a\leq n$.
\item[Case (3):] Assume that $(n+1)/2 < B < n$. For the sake of contradiction, suppose that $a>n$. Let $a=n+s$, where $s$ is even and $s\geq 2$. Since $(n+1)/2<B$ and $n$ is odd, $B=(n+t)/2$ for some odd $t$ with $t\geq 3$. Note that as $a<2B$, $s<t+1$. Since $s$ is even and $t$ is odd, $s\leq t-1$. Now, we have
    \begin{align*}
    a^2+n=BC \quad &\mbox{iff} \quad (n+s)^2+n=(n+t)C/2 \\
    \quad &\mbox{iff} \quad n^2+2sn+s^2+n = (n+t)C/2 \\
    \quad &\mbox{iff} \quad 2n^2+4sn+2s^2+2n =Cn+Ct \\
    \quad &\mbox{iff} \quad 2n^2 +(4s+2-C)n+(2s^2-Ct)=0.
    \end{align*}
Since $a^2+n=BC$, the last quadratic equation (in $n$) is solvable in $\Z$. So its discriminant $\triangle$ is a perfect square in $\Z$. But
\begin{align*}
\triangle &= (4s+2-C)^2-4(2)(2s^2-Ct) \\
&= 16s^2+4+C^2+16s-8sC-4C-16s^2+8Ct\\
&= C^2+(8t-8s-4)C+(16s+4).
\end{align*}
Since $\triangle$ is a perfect square in $\Z$, there is some $x\in \Z$ such that $\triangle =(C+x)^2$. So must have
$$2x=8t-8s-4 \quad \mbox{and} \quad x^2=16s+4.$$
Thus,
\begin{align*}
(4t-4s-2)^2&= 16s+4 \\
16t^2+16s^2+4-32st-16t+16s &= 16s+4 \\
16t^2+16s^2-32 st -16 t &= 0 \\
t^2+s^2-2st-t &=0 \\
(t-s)^2 &= t.
\end{align*}
So, $x=4t-4s-2$. As $t-1\geq s$, $x\geq 2$. This implies , in particular, that $C+x>0$ and so $\sqrt{\triangle}=C+x$. We thus have $n=\cfrac{-(4s+2-C)\pm (C+x)}{4}$. Considering the negative sign and substituting for $x$, we would have
$$n = \cfrac{-4s-2+C-C-4t+4s+2}{4}=-t<0,$$ a contradiction (as $n>0$). So we consider the positive sign to get
\begin{align*}
n&=\cfrac{-4s-2+C+C+4t-4s-2}{4}\\
&=\cfrac{2C+4t-8s-4}{4} \\
2n&=C+2t-4s-2.
\end{align*}
So, $C=2n-2t+4s+2$. Since $3a<2B+C$, it follows that $3n+3s<n+t+2n-2t+4s+2$ and thus $-t+s+2>0$. However, $s\leq t-1$ implies that $-t+s+2 \leq 1$. Combining the last two inequality yields that $-t+s+2=1$, and so $-t+s=-1$. Thus, $t=(t-s)^2=1$ contradicting the inequality $t\geq 3$. Hence, $a>n$ is false and, therefore, $a\leq n$ in this case.
\end{itemize}
\end{proof}

\begin{lem}\label{bound}
Let $(a,-b,-c) \in S_{a>0}(-n)$.
\begin{itemize}
\item[(i)] If $n$ is odd, then $a \leq n$.
\item[(ii)] If $n$ is even, then $a \leq n/2$.
\end{itemize}
\end{lem}

\begin{proof}
We shall only prove part (i), as the proof of part (ii) follows a similar argument. Assume that $n$ is odd. Fix $(a,-b,-c) \in S_{a>0}(-n)$, let $c=2c'$, and denote $|b|$, $|c|$, and $|c'|$ by $B$, $C$, and $C'$, respectively. We proceed in cases as follows.
\begin{enumerate}
\item[Case (1):] Suppose that $a\leq B$ and let $B=a+t$ for some $t\geq 0$. If $a>n$, then let $a=n+s$ for some $s\geq 1$. Then,
\begin{align*}
a^2+n=BC \quad &\Rightarrow \quad a^2+a-s=(a+t)C=aC+tC\\
\quad &\Rightarrow \quad a(a+1-C)=s+tC>0 \\
\quad &\Rightarrow \quad a+1-C >0 \\
\quad &\Rightarrow \quad a+1>C \\
\quad &\Rightarrow \quad a\geq C,
\end{align*}
a contradiction. So, $a\leq n$ in this case.
\item[Case (2):] Suppose that $a\leq C'$. If $a>n$, then let $a=n+s$ for some $s\geq 1$. Since $2a\leq 2C'=C$ in this case, let $C=2a+t$ for some $t\geq 0$. We then have
\begin{align*}
a^2+n=BC \quad &\Rightarrow \quad a^2+a-s=B(2a+t)=2aB+Bt \\
\quad &\Rightarrow \quad a(a+1-2B)=s+Bt >0 \\
\quad &\Rightarrow \quad a+1-2B >0 \\
\quad &\Rightarrow \quad a+1 >2B \\
\quad &\Rightarrow \quad a\geq 2B,
\end{align*}
a contradiction. So, $a\leq n$ in this case too.
\item[Case (3):] Suppose that $a>B$ and $a>C'$. Assume that $a>n$. It then follows from Lemma \ref{a,n} that $B>n$ and $C'>n$. We then have $n < B < a < 2B$ and $n < C' < a < C=2C'$. Let $a=n+1+s$, $B=n+1+t$, $C'=n+1+u$, for some $s,t,u\geq 0$. Since $B<a$ and $C'<a$, $t<s$, $u<s$, and $s\geq 1$. Now,
\begin{align*}
3a< 2B+C \quad &\Rightarrow \quad 3n+3+3s) < 2n+2+2t +n+1+u \\\
\quad &\Rightarrow \quad 3s<2t+u.
\end{align*}
But $2t+u < 2s+s=3s$. This is a contradiction. So, $a\leq n$ in this case as well.
\end{enumerate}
\end{proof}

\begin{prop}\label{orbit}
Under the action of $H$ on $\Q^*(\sqrt{-n})$, we have the following:
\begin{itemize}
\item[(i)] If $n$ is odd, then every orbit contains a unique totally negative quadruplet.
\item[(ii)] If $n$ is even, then every orbit contains either a unique pair of purely imaginary elements or a unique totally negative quadruplet (but not both).
\end{itemize}
\end{prop}

\begin{proof}
We first show that for any totally positive $\beta\in \Q^*(\sqrt{-n})$, $|a_{\beta}|<|a_{y^i(\beta)}|$ for all $i=1,2,3$. If $a_{\beta}, c_{\beta}>0$ (the case $a_\beta, c_\beta <0$ is similar), then (by Lemma \ref{table}) $$|a_{y(\beta)}| =|-a_{\beta}-c_\beta|=a_\beta+c_\beta>a_\beta=|a_\beta|.$$ In a similar manner we prove that $|a_{y^2(\beta)}|>|a_\beta|$ and $|a_{y^3(\beta)}|>|a_\beta|$.

Let $\al^H$ be an orbit in $\Q^*(\sqrt{-n})$ and, by Lemma \ref{positive}, let $\al_1\in \al^H$ be totally positive. So, by Lemma \ref{negative quadrable}, $x(\al_1)$ is totally negative. If $Q_{x(\al_1)} \in TN(-n)$, then we stop. Otherwise, for some $i=1,2,3$, $\al_2 = y^i x(\al_1)$ is either purely imaginary or totally positive. If $\al_2$ is purely imaginary, then $x(\al_2)$ is purely imaginary too (Lemma \ref{negative quadrable}) and we stop. Otherwise, $\al_2$ is totally positive and so $x(\al_2)$ is totally negative. If $Q_{x(\al_2)} \in TN(-n)$, then we stop. Else, for some $i=1, 2,3$, $\al_3=y^ix(\al_2)$ is either purely imaginary or totally positive. We repeat for $\al_3$ what we did for $\al_1$ and $\al_2$ to get either a pair of purely imaginary elements, a totally negative quadruplet, or carry on to get a totally positive $\al_4=y^ix(\al_3)$ for some $i=1,2,3$. We continue in this manner and suppose that no pair of purely imaginary elements pups up in the process, so we have a sequence of totally positive elements $\al_1, \al_2, \al_3, \dots$ such that, for every $j=1,2,3, \dots$, $\al_j = y^ix(\al_{j-1})$ for some $i=1,2,3$. By the argument at the beginning of this proof, we see that if $\al_2=yx(\al_1)$, then $y^3(\al_2)=x(\al_1)$ and so $$|a_{\al_2}|< |a_{y^3(\al_2)}|=|a_{x(\al_1)}|=|a_{\al_1}|.$$
In a similar manner, we have $|a_{\al_2}| < |a_{\al_1}|$ if $\al_2=y^2x(\al_1)$ or $\al_2=y^3x(\al_1)$. Doing the same for each $j=1, 2,3, \dots$, we have $$|a_{\al_1}|>|a_{\al_2}|>|a_{\al_3}|> \dots$$
Since $\{|a_{\al_i}|\}$ is a decreasing sequence of positive integers, it must terminate at some $|\al_k|$, say. We notice that $x(\al_k)$ is totally negative. If, for some $i=1, 2, 3$, $y^ix(\al_k)$ is purely imaginary, then we stop. Otherwise, $y^ix(\al_k)$ is totally negative for all $i=1,2,3$, as if any of $y^i x(\al_k)$ is totally positive, then $|a_{y^ix(\al_k)}|<|a_{x(\al_k)}|=|a_{\al_k}|$ contradicting the minimality of $|a_{\al_k}|$. By this we have reached at the totally negative quadruplet $Q_{x(\al_k)}$.

%For $n=2$, the only purely imaginary elements of $\Q^*(\sqrt{-2})$ are $\sqrt{-2}/2$ and $\sqrt{-2}/(-2)$ (see the proof of Lemma \ref{norm zero}). But, by Lemma \ref{fixed}, each of these two elements are fixed by $x$ and, by Lemma \ref{same sign}, each of them lies in a separate orbit. As for the totally negative quadruplet in this case, suppose that $\al=\cfrac{a+\sqrt{-2}}{c}\in \Q^*(\sqrt{-2})$ is such that $Q_\al\in TN(-n)$. Due to the bijection between $S_{a>0}(-2)$ and $S_{a<0}(-2)$, we can assume with no loss of generality that $(a,b,c)\in S_{a>0}(-2)$. By Lemmas \ref{a,n,b,c} and \ref{bound}, $a$ must be equal to 2. Since $c$ is even and $c$ divides $a^2+2=6$, $c\in \{\pm 2, \pm 6\}$. If $c= \pm 2$, then $a\nless |c|$, and if $c=\pm 6$, then $b=\pm 1$ and $a\nless 2|b|$. In any case, this contradicts the assumption that $(a,b,c)\in S_{a>0}(-2)$. Thus, there is no totally negative quadruplet under the action of $H$ on $\Q^*(\sqrt{-2})$.

To complete the proof, we need yet to tackle the uniqueness claim. Suppose that an orbit contains a totally negative quadruplet $Q_\al=\{\al, y(\al), y^2(\al), y^3(\al)\}$. We show that as we depart from $Q_\al$, every quadruplet we get must contain a totally positive element and three totally negative element and, thus, we can never reach another totally negative quadruplet in the orbit. The action of $y$ on $Q_\al$ obviously permutes the elements of $Q_\al$. So the only way to depart from $Q_\al$ is via the action of $x$. If we depart $Q_\al$ from the direction of $\al$, for instance, then we get $x(\al)$, which is then totally positive. Now the quadruplet $Q_{x(\al)}=\{x(\al), yx(\al), y^2x(\al), y^3x(\al)\}$ contains only one totally positive element, namely $x(\al)$, whereas the remaining elements are totally negative. The same scenario exactly occurs if we depart $Q_\al$ from the direction of either of $y^i(\al)$, for $i=1,2,3$, so we keep dealing with $Q_{x(\al)}$ with no loss of generality. Now, to depart from $Q_{x(\al)}$ in order to get another quadruplet, we note that if we depart from the direction of $x(\al)$ then we would land back in $Q_\al$ as $xx(\al)=\al$. So we should consider departing from the direction of $y^ix(\al)$, $i=1,2,3$, to get the totally positive $xy^ix(\al)$. Again, the quadruplet $Q_{xy^ix(\al)}$ contains only one totally positive element, namely $xy^ix(\al)$, whereas the remaining elements are totally negative. Continuing this process with no return from one quadruplet to its predecessor, we can never reach another totally negative quadruplet. This shows the uniqueness of a totally negative quadruplet in a single orbit, if it exists. The uniqueness of a pair of purely imaginary elements in a single orbit, if exists, follows a similar line of argument as above. This completes the proof.
\end{proof}

\vspace{.3cm}
In the following, we denote the set of purely imaginary pairs $\{\al, x(\al)\}$ of elements of $\Q^*(\sqrt{-n})$ by $PI(-n)$ for $n$ even. For $m \in \N$ and $k\in \N\cup \{0\}$ with $k\leq m$, we denote the number of positive divisors of $m$ which do not exceed $k$ by $d_{\leq k}(m)$.

Now, we stat the main theorem of this paper.

%\begin{cor}\label{no of orbits}
%The number of orbits in $\Q^*(\sqrt{-n})$ under the action of $H$ is equal to
%\begin{equation*}
%\left \{ \begin{array} {c@{\quad;\quad}l}
%\vspace{.3cm}
%|TN(-n)| & \mbox{if $n$ is odd} \\

%|PI(-n)|+|TN(-n)| & \mbox{if $n$ is even}
%\end{array} \right.
%\end{equation*}
%\end{cor}

%\begin{proof}
%This follows directly from Proposition \ref{orbit}.
%\end{proof}

\begin{thm}\label{main}
The number of orbits in $\Q^*(\sqrt{-n})$ under the action of $H$ is equal to
\begin{equation*}
\left \{ \begin{array} {c@{\quad;\quad}l}
\vspace{.3cm}
|TN(-n)| & \mbox{if $n$ is odd} \\

|PI(-n)|+|TN(-n)| & \mbox{if $n$ is even}
\end{array} \right.
\end{equation*}
\vspace{.5cm}
%\begin{equation*}
%\left \{ \begin{array} {c@{\quad;\quad}l}
%\vspace{.3cm}
%2 & \mbox{if $n=1$ or $n=2$} \\
%\vspace{.3cm}
%\cfrac{1}{2} \, |S_{a>0}(-n)| & \mbox{if $n$ is odd, $n\neq 1$}\\
%d(\cfrac{n}{2}) + \cfrac{1}{2}\, |S_{a>0}(-n)| & \mbox{if $n$ is even, $n\neq 2$}
%\end{array} \right.
%\end{equation*}
\begin{equation*}
=\left \{ \begin{array} {c@{\quad;\quad}l}
\vspace{.3cm}
2 & \mbox{if $n=1$ or $n=2$} \\
\vspace{.3cm}
\cfrac{1}{2} \, \mathop{\mathlarger{\sum}}_{\substack{i=1,\\ \mbox{\tiny odd}}}^n \left\{d(\cfrac{i^2+n}{2})-2\,d_{\leq \lfloor\frac{i}{2}\rfloor}(\cfrac{i^2+n}{2})-E_i(n)\right\} & \mbox{if $n$ is odd, $n\neq 1$}\\
d(\cfrac{n}{2}) + \cfrac{1}{2}\, \mathop{\mathlarger{\sum}}_{\substack{i=2,\\ \mbox{\tiny even}}}^{n/2} \left\{d(\cfrac{i^2+n}{2})-2\,d_{\leq \frac{i}{2}}(\cfrac{i^2+n}{2})-E_i(n) \right\} & \mbox{if $n$ is even, $n\neq 2$}
\end{array} \right.
\end{equation*}
where $E_i(n)$ is the number of positive even integers $C$ dividing $i^2+n$ and $\cfrac{2(i^2+n)}{C}+C \leq 3i$.
\end{thm}

\begin{proof}
The expression
\begin{equation*}
\left \{ \begin{array} {c@{\quad;\quad}l}
\vspace{.3cm}
|TN(-n)| & \mbox{if $n$ is odd} \\

|PI(-n)|+|TN(-n)| & \mbox{if $n$ is even}
\end{array} \right.
\end{equation*}
for the number of orbits in $\Q^*(\sqrt{-n})$ follows directly from Proposition \ref{orbit}, based on which we shall derive the second expression.

If $n=1$, the number of orbits in $\Q^*(\sqrt{-1})$ is $|TN(-1)|=2|TN_{a>0}(-1)|$. If $Q_\al \in TN_{a>0}(-1)$, then (as $(a_\al, -b_\al, -c_\al)\in S_{a>0}(-1)$), it follow from Lemma \ref{bound} that $0<a_\al \leq 1$ and so $a_\al=1$. But then $c_\al<0$ is an even divisor of $1^2+1=2$, that is $c_\al= -2$. Thus, $|TN_{a>0}(-1)|=1$ and so $|TN(-1)|=2$ as claimed. Note that this is an exceptional case of signatures in the sense that $|S_{a>0}(-1)|=1$ and not a multiple of 4 as expected. The reason for this is that the totally negative element $\al=(1+\sqrt{-1})/-2$ is fixed by $y$ and thus $Q_\al=\{\al\}$. Similarly, $Q_{\overline{\al}}=\{\overline{\al}\}$. In fact, $\al$ and $\overline{\al}$ are the only elements of $\C$ fixed by $y$, since $z\in \C$ is fixed by $y$ if and only if $-1/2(z+1)=z$, that is $z$ must be a solution of the quadratic equation $2z^2+2z+1=0$.

For the case $n=2$, the number of orbits in $\Q^*(\sqrt{-2})$ is $|PI(-2)|+|TN(-2)|$. Since \linebreak $\sqrt{-2}/c \in \Q^*(\sqrt{-2})$ if and only if $c$ is an even divisor of $2$, the only purely imaginary elements of $\Q^*(\sqrt{-2})$ are $\sqrt{-2}/2$ and $\sqrt{-2}/(-2)$. Note that the pair of purely imaginary elements containing $\sqrt{-2}/2$ is a singleton in the sense that it contains only $\sqrt{-2}/2$. Same thing is said about $\sqrt{-2}/(-2)$. The reason for this is that these two elements are fixed by $x$. In fact, they are the only elements of $\C$ fixed by $x$, since $z\in \C$ is fixed by $x$ if and only if $-1/2z=z$, that is $z$ must be a solution of the quadratic equation $2z^2+1=0$. Since, by Lemma \ref{same sign}, these two elements lie in distinct orbits, we have $|PI(-2)|=2$. On the other hand, if $Q_\al \in TN_{a>0}(-2)$, then (as $(a_\al, -b_\al, -c_\al)\in S_{a>0}(-2)$), it follows from Lemma \ref{bound} that $0<a_\al\leq 1$ and so $a_\al=1$. But then $c_\al<0$ is an even divisor of $1^2+2=3$, which is absurd. Thus, $TN_{a>0}(-2)=\varnothing$ and, therefore, $TN(-2)=\varnothing$. This shows that the number of orbits in $\Q^*(\sqrt{-2})$ is 2.

Now, suppose that $n\not\in\{1,2\}$. Let $Q_\al \in TN_{a>0}(-n)$. Since $Q_\al$ is stable under the action of $<y>$, $S_{a>0}(-n)$ contains the signature of $\al$ if and only if it contains the signatures of $y(\al), y^2(\al),$ and $y^3(\al)$. This shows that we can associate to every $\displaystyle{Q_\al \in TN_{a>0}(-n)}$ four distinct elements of $S_{a>0}(-n)$. Moreover, if $Q_\al, Q_\beta \in TN_{a>0}(-n)$ are not equal, then they are disjoint indeed. Thus, $|S_{a>0}(-n)|=4|TN_{a>0}(-n)|$. Since $|TN(-n)|=2|TN_{a>0}(-n)|$, we get
$|TN(-n)|=\cfrac{1}{2}\,|S_{a>0}(-n)|$. To ease notations in the calculations below, we denote an element $(a,-b,-c)\in S_{a>0}(-n)$ by $(a,B,C)$ (so, $B=-b$ and $C=-c$). Considering this notation, recall that
$$S_{a>0}(-n)=\{(a,B,C)\in \N^3\,|\, \mbox{$C$ is even}, a^2+n=BC, a<2B, a<C, 3a<2B+C\}.$$
Due to Lemmas \ref{a,n,b,c} and \ref{bound}, we partition $S_{a>0}(-n)$ as follows:
\begin{equation*}
S_{a>0}(-n)= \left \{ \begin{array} {c@{\quad;\quad}l}
\vspace{.3cm}
\mathop{\mathlarger{\bigcup}}_{\substack{i=1,\\ \mbox{\tiny odd}}}^{n} S^i_{a>0}(-n) & \mbox{if $n$ is odd, $n\neq 1$} \\
\mathop{\mathlarger{\bigcup}}_{\substack{i=2,\\ \mbox{\tiny even}}}^{n/2} S^i_{a>0}(-n) & \mbox{if $n$ is even, $n\neq 2$},
\end{array} \right.
\end{equation*}
where
$$S^i_{a>0}(-n)=\{(i,B,C)\in \N^3 \,|\, \mbox{$C$ is even}, i^2+n=BC, i<2B, i<C, 3i<2B+C\}.$$
Since the partitioning sets are mutually disjoint, the cardinality of $S_{a>0}(-n)$ is equal to the sum of the cardinalities of its partitioning sets. For a fixed $i$, we see that an element $(i, B,C)\in S^i_{a>0}(-n)$ can also be written in the form $(i,\cfrac{i^2+n}{C}, C)$. So we me write $S^i_{a>0}(-n)$ in the form:
$$S^i_{a>0}(-n)=\{(i,\cfrac{i^2+n}{C},C)\,|\, \mbox{$0<C$ is even}, C|(i^2+n), i<\cfrac{2(i^2+n)}{C}, i<C, 3i<\cfrac{2(i^2+n)}{C}+C\}.$$
On the other hand, we observe that
$$S^i_{a>0}(-n)=S^{i,*}_{a>0}(-n) - \left \{ S^{i, 2B}_{a>0}(-n) \cup S^{i,C}_{a>0}(-n) \cup S^{i, 2B+C}_{a>0}(-n)\right \},$$
where
\begin{align*}
S^{i,*}_{a>0}(-n)&=\left\{\left(i,\cfrac{i^2+n}{C},C\right)\,|\, \mbox{$0<C$ is even}, C|(i^2+n)\right\},\\
S^{i,2B}_{a>0}(-n)&=\left\{\left(i,\cfrac{i^2+n}{C},C\right)\,|\, \mbox{$0<C$ is even}, C|(i^2+n), \cfrac{2(i^2+n)}{C}\leq i \right\}\\
S^{i,C}_{a>0}(-n)&=\left\{\left(i,\cfrac{i^2+n}{C},C\right)\,|\, \mbox{$0<C$ is even}, C|(i^2+n), C\leq i\right\}\\
S^{i, 2B+C}_{a>0}(-n)&=\left\{\left(i,\cfrac{i^2+n}{C},C\right)\,|\, \mbox{$0<C$ is even}, C|(i^2+n), \cfrac{2(i^2+n)}{C}+C\leq 3i\right\}.
\end{align*}
We now claim that the sets $S^{i,2B}_{a>0}(-n)$, $S^{i,C}_{a>0}(-n)$, and $S^{i, 2B+C}_{a>0}(-n)$ are mutually disjoint. Indeed, if $S^{i,2B}_{a>0}(-n)\bigcap S^{i,C}_{a>0}(-n)\neq \varnothing$, then there is an even positive integer $C$ such that $C|(i^2+n)$, $\displaystyle{\cfrac{2(i^2+n)}{C}\leq i}$, and $C\leq i$. This implies that $\cfrac{2(i^2+n)}{i} \leq i$ and so $2(i^2+n)\leq i^2$, which is absurd. Also, if $S^{i,2B}_{a>0}(-n)\bigcap S^{i,2B+C}_{a>0}(-n)\neq \varnothing$, then there is an even positive integer $C$ such that $C|(i^2+n)$, $\cfrac{2(i^2+n)}{C}\leq i$, and $\cfrac{2(i^2+n)}{C}+C\leq 3i$. Let $0\leq t < i$ be such that $\cfrac{2(i^2+n)}{C}+t=i$. Then, $2(i^2+n)+Ct=iC$ and so $C=\cfrac{2(i^2+n)}{i-t}$. Substituting this in the inequality $\cfrac{2(i^2+n)}{C}+C\leq 3i$ yields
\begin{align*}
(i-t)^2+2(i^2+n)&\leq 3i(i-t)\\
i^2-2it+t^2+2i^2+2n &\leq 3i^2-3it\\
2n+t^2 &\leq it \leq 0,
\end{align*}
which is absurd too. Similarly, if $S^{i,C}_{a>0}(-n)\bigcap S^{i,2B+C}_{a>0}(-n)\neq \varnothing$, then there is an even positive integer $C$ such that $C|(i^2+n)$, $C\leq i$, and $\cfrac{2(i^2+n)}{C}+C\leq 3i$. Let $0\leq t <i$ be such that $C+t=i$. Substituting the value of $C$ in the inequality $\cfrac{2(i^2+n)}{C}+C\leq 3i$ yields
\begin{align*}
2(i^2+n)+(i-t)^2&\leq 3i(i-t) \\
2i^2+2n+i^2+t^2-2it &\leq 3i^2-3it\\
2n+t^2&\leq it \leq 0,
\end{align*}
which is absurd as well. By this, the claim is settled. Hence, we have
$$|S^i_{a>0}(-n)=|S^{i,*}_{a>0}(-n)|-|S^{i,2B}_{a>0}(-n)|-|S^{i,C}_{a>0}(-n)|-|S^{i, 2B+C}_{a>0}(-n)|.$$
It is obvious that $|S^{i,*}_{a>0}(-n)|$ is equal to the number of positive {\it even} divisors of $i^2+n$, which is, by elementary number theory, equal to the number of positive divisors of $\cfrac{i^2+n}{2}$. So, \linebreak $|S^{i,*}_{a>0}(-n)| = d\left (\cfrac{i^2+n}{2}\right )$. As for the cardinality $|S^{i,C}_{a>0}(-n)|$, we observe that it is equal to the number of positive even divisors of $i^2+n$ which do not exceed $i$; that is $d_{\leq i}(i^2+n)$. Again, by elementary number theory, $d_{\leq i}(i^2+n) = d_{\leq \frac{i}{2}}(\cfrac{i^2+n}{2})$. On the other hand, it can be easily checked that the map $S^{i,2B}_{a>0}(-n) \to S^{i,C}_{a>0}(-n)$ defined by $\left(i,\cfrac{i^2+n}{C},C\right) \mapsto \left(i, \cfrac{C}{2}, \cfrac{2(i^2+n)}{C}\right)$ is bijective and, thus, $|S^{i,2B}_{a>0}(-n)|=|S^{i,C}_{a>0}(-n)|=d_{\leq \frac{i}{2}}(\cfrac{i^2+n}{2})$. Finally, it is obvious that $|S^{i, 2B+C}_{a>0}(-n)|=E_i(n)$, and by this the proof is complete.
%The last cardinality we need to consider is $|S^{i, 2B+C}_{a>0}(-n)|$. Although we do not have a formula that computes this cardinality, we find some bounds that make finding all elements of $S^{i, 2B+C}_{a>0}(-n)$ an easier brute-force job. Let $\left(i,\cfrac{i^2+n}{C},C\right) \in S^{i, 2B+C}_{a>0}(-n)$. We show first that $i<C<2i$. Since $S^{i, 2B+C}_{a>0}(-n)$ and $S^{i, C}_{a>0}(-n)$ are disjoint, $C>i$. Suppose that $C\geq 2i$. Then $\cfrac{2(i^2+n)}{C}+C \leq 3i$ implies that $\cfrac{2(i^2+n)}{C}+2i\leq 3i$ and so $\cfrac{2(i^2+n)}{C}\leq i$, which is impossible since $S^{i, 2B+C}_{a>0}(-n)$ and $S^{i, 2B}_{a>0}(-n)$ are disjoint. Thus, $C<2i$. % We now show that $C\geq i+3$. Suppose that $C=i+t$ for $t\in \{1,2\}$. Then $\cfrac{2(i^2+n)}{i+t} + i+t \leq 3i$ yields $2i^2+2n+i^2+2ti+t^2 \leq 3i^2 + 3ti$ and so $2n+t^2 \leq ti$. Since $i\leq n$ (by Lemma \ref{bound}), we get $2n+t^2 \leq tn$, which is absurd. Thus, $C\geq i+3$. Similarly, we show that $C\leq 2i-3$. Suppose that $C=2i-t$ for $t\in \{1,2\}$. Then $\cfrac{2(i^2+n)}{2i-t}+2i-t\leq 3i$ yields $2i^2+2n+4i^2+t^2-4ti\leq 6i^2-3ti$ and so $2n+t^2\leq ti\leq tn$, which is absurd too. Thus, $C\leq 2i-3$.
\end{proof}

\begin{cor}\label{transitive}
The action of $H$ on $\Q^*(\sqrt{-n})$ is intransitive for every $n$.
\end{cor}

\begin{proof}
For $n=1$ and $n=2$, the number of orbits is 2 and so the action is intransitive in these two cases. For an even $n$ different from 2, $d(n/2) \geq 2$ and so the action is intransitive in this case as well. If $n$ is odd and different from 1, we can check that
$$\left\{ \left(1, \cfrac{n+1}{2}, 2\right), \left(1,1,n+1\right), \left(n, \cfrac{n+1}{2}, 2n\right), \left(n,n,n+1\right)\right\}\subseteq S_{a>0}(-n).$$
Thus, $|TN(-2)|=\cfrac{1}{2}\,|S_{a>0}(-n)| \geq \cfrac{1}{2}\cdot 4=2$ and, therefore, the action is also intransitive in this case.
\end{proof}

%%%%%%%%%%%%%%%%%%%%%%%%%%%%%%%%%%%%%%%%%%%%%%%%%%%%%%%%%%%%%%%%%%%%%%%%%%

\section{Examples}
To compute the number of orbits based on Theorem \ref{main}, we remark that the first formula (for $n\neq 1,2$) requires exhaustive calculation of the sets $PI(-n)$ and $S_{a>0}(-n)$; whereas the second formula requires the calculation of elements leading to $E_i(n)$. We shall follow both approaches below.

We know that for an even $n$, purely imaginary elements always exist. In the examples below, however, we chose to consider $n=10$ and $n=14$ to illustrate a case for an even $n$ in the former case where totally negative quadruplets do not exist in contrast to the latter case where they exist. On the other hand, we know that for an odd $n$, purely imaginary elements do not exist while totally negative quadruplets always exist. However, we chose to consider below $n=11$ and $n=15$ to illustrate a case for an odd $n$ in the former case where all $E_i(n)$ are zero in contrast to the latter case where not all $E_i(n)$ are zero.

\begin{example}{\underline{$n=10$}}$\\$
\indent By Theorem \ref{main}, the number of orbits under the action of $H$ on $\Q^*(\sqrt{-10})$ is $$d(5)+\cfrac{1}{2}\, |S_{a>0}(-10)|=2+\cfrac{1}{2}\, \mathop{\mathlarger{\sum}}_{\substack{i=2,\\ \mbox{\tiny even}}}^{5} |S^i_{a>0}(-10)|.$$
For $i=2$, $i^2+n=14$ and the only possible values of $C$ are $2$ and $14$. But, if $C=2$ then $i\nless C$ and if $C=14$ then $B=1$ and $i\nless 2B$. So, $|S^2_{a>0}(-10)| =\varnothing$. For $i=4$, $i^2+n=26$ and the only possible values of $C$ are $2$ and $26$. Similar to the argument above, we get $|S^4_{a>0}(-10)|=\varnothing$. Thus, $|S_{a>0}(-10)|=0$ and the number of orbits in $\Q^*(\sqrt{-10})$ is 2.

On the other hand, the other formula given by Theorem \ref{main} gives
$$d(5) + \cfrac{1}{2}\, \mathop{\mathlarger{\sum}}_{\substack{i=2,\\ \mbox{\tiny even}}}^{5} \left\{d(\cfrac{i^2+n}{2})-2\,d_{\leq \frac{i}{2}}(\cfrac{i^2+n}{2})-E_i(n)\right\}.$$
As above, for $i=2$, if $C=2$ then $B=7$ and $2\cdot 7+2=16 \nleq 3\cdot i =6$, and if $C=14$ then $B=1$ and $2\cdot 1 + 14 =16 \nleq 3\cdot i =6$. So $E_2(10)=0$. Similarly, $E_4(10)=0$. Thus, the number of orbits in $\Q^*(\sqrt{-10})$ is equal to
$$2+\cfrac{1}{2}\,\{(d(7)-2d_{\leq 1}(7)-0)+(d(13)-2d_{\leq 2}(13)-0)\}= 2+\cfrac{1}{2}\,\{(2-2)+(2-2)\}=2.$$
\end{example}

\begin{example}{\underline{$n=14$}}$\\$
\indent Similar to the case $n=10$, it can be verified that $S^2_{a>0}(-14)=\{(2,3,6)\}$, \linebreak $S^4_{a>0}(-14)=\{(4,5,6),(4,3,10)\}$, and $S^6_{a>0}(-14)=\{(6,5,10)\}$. So, the number of orbits in $\Q^*(\sqrt{-14})$ under the action of $H$ is equal to
$$d(7)+\cfrac{1}{2}\, |S_{a>0}(-14)|=2+\cfrac{1}{2}\, \mathop{\mathlarger{\sum}}_{\substack{i=2,\\ \mbox{\tiny even}}}^{7} |S^i_{a>0}(-14)|=2+\cfrac{1}{2}\,(1+2+1)=2+2=4.$$

On the other hand, using the other formula given by Theorem \ref{main}, it can be checked that $E_i(14)=0$ for all $i\in \{2,4,6\}$. Thus, the number of orbits in $\Q^*(\sqrt{-14})$ is equal to
\begin{align*}
&d(7) + \cfrac{1}{2}\, \mathop{\mathlarger{\sum}}_{\substack{i=2,\\ \mbox{\tiny even}}}^{7} \left\{d(\cfrac{i^2+n}{2})-2\,d_{\leq \frac{i}{2}}(\cfrac{i^2+n}{2})-E_i(n)\right\}\\
&\qquad =2+\cfrac{1}{2}\,\{(d(9)-2d_{\leq 1}(9)-0)+(d(15)-2d_{\leq 2}(15)-0)+(d(25)-2d_{\leq 3}(25)-0)\}\\
&\qquad= 2+\cfrac{1}{2}\,\{(3-2\cdot 1)+(4-2\cdot 1)+(3-2\cdot 1)\}\\
&\qquad=2+2=4.
\end{align*}
\end{example}

\begin{example}{\underline{$n=11$}}$\\$
\indent By Theorem \ref{main}, the number of orbits under the action of $H$ on $\Q^*(\sqrt{-11})$ is $$\cfrac{1}{2}\, |S_{a>0}(-11)|=\cfrac{1}{2}\, \mathop{\mathlarger{\sum}}_{\substack{i=1,\\ \mbox{\tiny odd}}}^{11} |S^i_{a>0}(-11)|.$$
For $i=1$, $i^2+n=12$ and the only possible values of $C$ are $2, 4, 6,$ and $12$. The corresponding values of $B$ are, respectively, $6, 3, 2,$ and $1$. It can be checked that for every such triplet $(i,C,B)$, the conditions $i<C$, $i<2B$, and $3i<2B+C$ are satisfied. Thus,
$$S^1_{a>0}(-11)=\{(1,6,2), (1,3,4), (1,2,6), (1,1,12)\}.$$
Similarly, the following can be verified:
\begin{align*}
S^3_{a>0}(-11)&=\{(3,5,4), (3,2,10)\},\\
S^5_{a>0}(-11)&=\{(5,6,6), (5,3,12)\},\\
S^7_{a>0}(-11)&=\{(7,6,10), (7,5,12)\},\\
S^9_{a>0}(-11)&=\varnothing,\\
S^{11}_{a>0}(-11)&= \{(11,11,12),(11,6,22)\}.
\end{align*}
Thus, the number of orbits in $\Q^*(\sqrt{-11})$ is $\cfrac{1}{2}\,(4+2+2+2+0+2)=6$.

On the other hand, the other formula given by Theorem \ref{main} gives
$$\cfrac{1}{2}\, \mathop{\mathlarger{\sum}}_{\substack{i=1,\\ \mbox{\tiny odd}}}^{11} \left\{d(\cfrac{i^2+n}{2})-2\,d_{\leq \lfloor\frac{i}{2}\rfloor}(\cfrac{i^2+n}{2})-E_i(n)\right\}.$$
As above, for $i=1$, it can be checked that $2B+C \nleq 3\cdot i=3$ for every possible pair $(B,C)$. So $E_1(11)=0$. It can also be similarly checked that $E_i(11)=0$ for every $i\in\{3,5,7,9,11\}$. Thus, the number of orbits in $\Q^*(\sqrt{-11})$ is equal to
\begin{align*}
&\cfrac{1}{2}\, \{(d(6)-2d_{\leq 0}(6)-0)+(d(10)-2d_{\leq 1}(10)-0)+(d(18)-2d_{\leq 2}(18)-0)\\
& \quad +(d(30)-2d_{\leq 3}(30)-0)+(d(46)-2d_{\leq 4}(46)-0)+(d(66)-2d_{\leq 5}(66)-0)\}
\\&=\cfrac{1}{2}\,\{(4-0)+(4-2)+(6-4)+(8-6)+(4-4)+(8-6)\}=\cfrac{1}{2}\cdot 12=6.
\end{align*}
\end{example}

\begin{example}{\underline{$n=15$}}$\\$
\indent Similar to the case $n=11$, it can be verified that
\begin{align*}
S^1_{a>0}(-15)&=\{(1,8,2), (1,4,4), (1,2,8), (1,1,16)\}\\
S^3_{a>0}(-15)&=\{(3,6,4), (3,4,6), (3,3,8), (3,2,12)\}\\
S^5_{a>0}(-15)&=\{(5,5,8), (5,4,10)\}\\
S^7_{a>0}(-15)&=\{(7,8,8),(7,4,16)\}\\
S^9_{a>0}(-15)&=\{(9,8,12),(9,6,16)\}\\
S^{11}_{a>0}(-15)&= S^{13}_{a>0}(-15)=\varnothing \\
S^{15}_{a>0}(-15)&=\{(15,15,16), (15,8,30)\}.
\end{align*}
Thus, by Theorem \ref{main}, the number of orbits in $\Q^*(\sqrt{-15})$ is
$$\cfrac{1}{2}\, |S_{a>0}(-15)|=\cfrac{1}{2}\, \mathop{\mathlarger{\sum}}_{\substack{i=1,\\ \mbox{\tiny odd}}}^{15} |S^i_{a>0}(-15)|=\cfrac{1}{2}\,(4+4+2+2+2+0+0+2)=8.$$

On the other hand, using the other formula given in Theorem \ref{main}, it can be verified that $\displaystyle{E_i(15)=0}$ for all $i\in \{1,3,5,7,9,11,13\}$. However, for $i=15$, each of the two triplets, and none else, $(15,10,24)$ and $(15,12,20)$ satisfy the condition $2B+C\leq 3i$. So, $E_{15}(15)=2$. Thus, the number of orbits in $\Q^*(\sqrt{-15})$ is equal to
\begin{align*}
&\cfrac{1}{2}\, \{(d(8)-2d_{\leq 0}(8)-0)+(d(12)-2d_{\leq 1}(12)-0)+(d(20)-2d_{i\leq 2}(20)-0)\\
&+ (d(32)-2d_{i\leq 3}(32)-0)+(d(48)-2d_{i\leq 4}(48)-0)=(d(68)-2d_{i\leq 5}(68)-0)\\
&+ (d(92)-2d_{i\leq 6}(68)-0)+(d(120)-2d_{i\leq 7}(120)-2)\}\\
&=\cfrac{1}{2}\,\{(4-0)+(6-2)+(6-4)+(6-4)+(10-8)+(6-6)+(6-6)+(16-12-2)\}\\
&=8.
\end{align*}
\end{example}

\section*{Acknowledgement}
The author would like to express his gratitude to King Khalid University for providing administrative and technical support.

\end{document}